\documentclass[reqno]{amsart}

\usepackage{amsmath, amssymb}
\usepackage{amsthm}
\usepackage{enumitem}
\usepackage{mathrsfs}
\usepackage{tikz}
\usepackage{tikz-3dplot}

\usetikzlibrary{intersections, calc}
\usetikzlibrary{shapes.geometric}

\usepackage{hyperref}
\hypersetup{colorlinks=true,urlcolor=blue,linkcolor=blue,citecolor=blue}

\usepackage{url}

\setlist[enumerate]{label={\upshape (\roman*)}}

\makeatletter
	\@addtoreset{equation}{section}

\theoremstyle{plain}
\newtheorem{theorem}{\indent\sc Theorem}[section]
\newtheorem{lemma}[theorem]{\indent\sc Lemma}
\newtheorem{corollary}[theorem]{\indent\sc Corollary}
\newtheorem{proposition}[theorem]{\indent\sc Proposition}

\theoremstyle{definition}
\newtheorem{definition}[theorem]{\indent\sc Definition}
\newtheorem{remark}[theorem]{\indent\sc Remark}
\newtheorem{example}[theorem]{\indent\sc Example}

\newcommand{\cE}{\mathcal{E}} 
\newcommand{\cF}{\mathcal{F}}
\newcommand{\G}{\mathbb{G}}
\newcommand{\E}{\mathbb{E}}
\newcommand{\bH}{\mathbf{H}}
\newcommand{\Ao}{A_{\omega}}
\newcommand{\Bo}{B_{\omega}}
\begin{document}

\title[Markov properties for Gaussian fields]{Markov properties for Gaussian fields associated with Dirichlet forms}
\author[T. Ooi]{Takumu Ooi}\thanks{Research Institute for Mathematical Sciences, Kyoto University, Kyoto, 606-8502, JAPAN. JSPS Research Fellow (DC1).\ E-mail:ooitaku@kurims.kyoto-u.ac.jp}
\begin{abstract}
We prove the equivalence of the local property for an irreducible regular Dirichlet form and the Markov property for the Gaussian field associated with the Dirichlet form. Moreover we introduce a strong Markov property for Gaussian fields and present some sufficient conditions for this to hold.
\end{abstract}

\maketitle

\section{Introduction}\label{intro}
The Markov property is an important property for stochastic processes. Intuitively, it is the property that the future state of the process is independent of the past state given the present state. The Markov property has also been introduced for various other stochastic models. In studies of random fields, for example, these include L\'{e}vy's \(n\)-parameter Brownian motion \cite{Mc}, Gaussian processes with a multidimensional parameter \cite{P}, and random fields associated with reproducing kernel Hilbert spaces \cite{KM}.

Markov properties for Gaussian free field (GFF, in abbreviation) has been used in literatures. Nelson (\cite{N}) proved the (massive) GFF enjoys the Markov property. For a domain \(D\subset \mathbb{R}^d\), the GFF on \(D\) is the complete linear space of Gaussian random variables indexed by \(H_0^1(D)\) whose means are zero and covariances are given by the Dirichlet inner products, where \(H_0^1(D)\) is the completion of the space of all continuous functions with compact support in \(D\) by the Dirichlet inner product. A massive GFF is what is obtained when the covariance is given by the sum of the Dirichlet inner product and the constant multiple of \(L^2\)-inner product. GFFs can be considered on many other spaces and play an important role in the theory of random surfaces, quantum field theory, statistical physics, and amongst other areas. See \cite{BP} and \cite{S} for details. The Markov property for GFFs is applied in these areas, for example, to establish a reflection positivity also known as Osterwalder-Schrader positivity in physics (\cite{N2}), to establish a sewing operation for manifolds (\cite{Di}), and to construct a coupling between GFFs and occupation times of random interlacements on graphs (\cite{Sz}). Another property that is often considered is the domain Markov property (\cite{DS},\cite{HMP},\cite{SS}, etc.). This property is different from the Markov property treated in this paper, and we will note the relation between them in Remark \ref{Fukurem}.

As noted above, the covariance of the GFF is given by the Dirichlet inner product, which is a regular Dirichlet form. The definition of a GFF may be generalized to Gaussian fields whose covariances are Dirichlet forms. Before we consider these Gaussian fields, we give definitions and introduce some properties of Dirichlet forms.

Let \(E\) be a locally compact separable metric space and \(m\) be a positive Radon measure satisfying \({\rm supp}(m)=E\). Throughout this paper, for a given Borel set \(A\), \(C_c(A)\) denotes the family of all continuous functions with compact support contained in \(A\). Moreover, \(a\wedge b\) means \(\min \{a,b\}\) and \(a\vee b\) means \(\max \{a,b\}\).
\begin{definition}
Let \(\cF\) be a dense linear subspace of \(L^2(E;m)\) and \(\cE:\cF\times \cF\to \mathbb{R}\) be a non-negative definite symmetric bilinear form. The pair \((\cE,\cF)\) is called a \emph{Dirichlet form} on \(L^2(E;m)\) if the following conditions hold.\\
\((1)\) The space \(\cF\) is complete with respect to the norm \(\sqrt{\cE_1(\cdot, \cdot)}\), where \(\cE_1\) is the sum of \(\cE\) and the \(L^2(E;m)\) inner product.\\
\((2)\) For any \(f\in \cF\), it holds that \(0\vee f \wedge 1\in \cF \) and \(\cE(0\vee f \wedge 1,0\vee f \wedge 1)\leq \cE(f,f)\).\\
Moreover, \((\cE,\cF)\) is \emph{regular} if \(C_c(E)\cap \cF\) is \(\cE_1\)-dense in \(\cF\) and uniformly dense in \(C_c(E)\).
\end{definition}

Let \((\cE, \cF)\) be a regular Dirichlet form on \(L^2(E;m)\). Then there exists an \(m\)-symmetric Markov process \((\{Z_t\}_{t\geq 0}, \{\mathbb{P}_x\}_{x\in E})\) on \(E\) associated with \((\cE,\cF)\) (see \cite{FOT} or \cite{CF}, for details). We say \((\cE, \cF)\) is \emph{irreducible} (resp.\ \emph{transient, recurrent}) if its associated process is irreducible (resp.\ transient, recurrent).

\begin{definition}
A Dirichlet form \((\cE,\cF)\) on \(L^2(E;m)\) is \emph{local} if \(\cE(u,v)=0\) for all \(u,v \in \cF\) having disjoint compact support. 
\end{definition}
It is known that a regular Dirichlet form is said to be local if and only if its associated Markov process has continuous paths, namely the process is a diffusion (see \cite{FOT} or \cite{CF}, for details).\\

We now define \emph{Gaussian fields associated with Dirichlet forms}, which are the main objects in this paper. For a regular Dirichlet form \((\cE,\cF)\) on \(L^2(E;m)\), there exists a Gaussian field \(\G(\cE):=\{X_f\}_{f\in \cF_e}\) defined on a probability space \((\Omega, \mathscr{M}, \mathbb{P})\) satisfying \(\E(X_f)=0\) and \(\E(X_fX_g)=\cE(f,g)\) for \(f,g\in \cF_e\), see \cite{FO2} and \cite{Dud}. Here, \(\cF_e\) is the extended Dirichlet space of \(\cF\), more precisely, \(\cF_e\) is defined as the family of functions \(f:E\to \mathbb{R}\) having an \(\cE\)-Cauchy sequence \(\{f_n\}\subset \cF\) such that \(f_n\to f\) \(m\)-a.e. We call \(\G(\cE)\) the Gaussian field on \((\Omega,\mathscr{M},\mathbb{P})\) associated with the Dirichlet form \((\cE,\cF)\) on \(L^2(E;m)\). 

Throughout this paper, we may assume that \(\mathscr{M}\) is complete and \(\mathscr{N}\) denotes the collection of all \(\mathbb{P}\)-null sets. For \(\G(\cE)\), we define the \(\sigma\)-fields that will correspond to the filtrations of the standard Markov property for stochastic processes, and the Markov property as follows. 

\begin{definition}
For \(f\in \cF_e\), we define the \emph{spectrum} \(s(f)\) of \(f\) as the complement of the largest open set \(U\) satisfying \(\cE(f,g)=0\) for all \(g\in C_c(E)\cap \cF\) with \({\rm supp}(g)\subset U\).

For \(A\subset E\), we set \(\sigma(A)\subset \mathscr{M}\) to be the \(\sigma\)-field generated by \(\{X_f: f\in \cF_e,\ s(f)\subset A\}\) and \(\mathscr{N}\).
\end{definition}

\begin{definition}
For a set \(A\subset E\), \(\G(\cE)\) has \emph{the Markov property} with respect to \(A\) if \(\sigma (\overline{A^c})\) is conditionally independent of \(\sigma (\overline{A})\) given \(\sigma (\partial A)\), which means that \(\mathbb{P}(\Gamma \cap \Sigma|\sigma(\partial A))=\mathbb{P}(\Gamma |\sigma(\partial A))\mathbb{P}(\Sigma|\sigma(\partial A))\) for \(\Gamma \in \sigma(\overline{A})\) and \(\Sigma \in \sigma(\overline{A^c}).\)
\end{definition}

The following theorem is the first main theorem of this paper concerning the Markov property.

\begin{theorem}\label{mainmain}
Let \((\cE,\cF)\) be an irreducible regular Dirichlet form on \(L^2(E;m)\). Then the following are equivalent:\\
\((1)\) The Dirichlet form \((\cE,\cF)\) is local;\\
\((2)\) The Gaussian field \(\G(\cE)\) has the Markov property with respect to any relatively compact open set;\\
\((3)\) The Gaussian field \(\G(\cE)\) has the Markov property with respect to any open set;\\
\((4)\) The Gaussian field \(\G(\cE)\) has the Markov property with respect to any subset of \(E\).
\end{theorem}

This equivalence was proved by Dynkin \cite{D} when \((\cE,\cF)\) is transient under an absolutely continuous condition \((AC)\), which is the condition that there exists a Borel properly exceptional set \(N\subset E\) such that the transition semigroup \(P_t(x,\cdot)\) is absolutely continuous with respect to \(m\) for each \(t>0\) and \(x\in N^c\). Moreover it was proved by R\"{o}ckner \cite{R} when \((\cE,\cF)\) is transient without the assumption \((AC)\). When \((\cE,\cF)\) is recurrent, under the assumption \((AC)\), Fukushima and Oshima \cite{FO} proved the equivalence between \((1),(2)\) and \((3)\). Moreover, in \cite{FO2}, it is proved that \((3)\) implies \((1)\) when \((\cE,\cF)\) is recurrent without the assumption \((AC)\) (in fact, it holds that \((2)\) implies \((1)\) by using their proof).\\
So, in order to complete the proof of Theorem \ref{mainmain}, it is enough to show that \((1)\) implies \((3)\) and \((1)\) is equivalent to \((4)\) when \((\cE,\cF)\) is recurrent without the assumption \((AC\)). We prove these in Section \ref{Markov}.\\

We next consider the strong Markov property. The strong Markov property for stochastic processes has been discussed extensively in the literatures. Moreover strong Markov properties for set-indexed processes where stopping times are replaced by ``stopping sets" are studied in \cite{B, E1, E2, EO, K}.
For the case of the Gaussian field associated with a Dirichlet form \((\cE,\cF)\), if \((\cE, \cF)\) is strongly recurrent, \(\G(\cE)\) can be viewed as a set-indexed field. Hence we can apply the strong Markov property of the set-indexed process to \(\G(\cE)\) in this case. Indeed, Sznitman \cite{Sz} uses the strong Markov property for GFFs on graphs in order to construct a coupling between GFFs and the occupation times of random interlacements on graphs.
However, if \((\cE, \cF)\) is not strongly recurrent, such as in the case of GFFs on sets of \(\mathbb{R}^d\) for \(d\geq 2\), we can not define \(\G(\cE)\) as a random function. Hence we can not apply the strong Markov property for the set-indexed process to \(\G(\cE)\) in general. So we will introduce the strong Markov property for \(\G(\cE)\) on \((\Omega, \mathscr{M}, \mathbb{P})\) and state some sufficient conditions for the strong Markov property to hold.

Recall that \(\G(\cE)\) has the Markov property if \(\sigma(\overline{A})\) is conditionally independent of \(\sigma(\overline{A^c})\) given \(\sigma(\partial A)\). We will generalize these conditions to random sets.

Let \(F(E)\) be the collection of all closed sets in \(E\). We consider a \(\sigma\)-field on \(F(E)\) given by \(\sigma(\{A\in F(E):A\subset U\}:U\subset E{\rm \ is \ open})\). Next we introduce a random set \(\Ao:\Omega \to F(E)\), which is measurable. We call \(\Ao\) a \emph{stopping set} if \(\Ao\) is a random set that satisfies \(\{\Ao\subset A\}\in \sigma(A)\) for any closed set \(A.\) We define filtrations of stopping sets and the strong Markov property as follows.

\begin{definition}\label{defstrongfil}
For random sets \(\Ao, \Bo\) satisfying \(\Bo\subset \Ao\), if \(\Ao\) and \(\overline{{\Bo}^c}\) are stopping sets, we define \(\sigma\)-fields as follows.\\
\[\sigma(\Ao):=\sigma(\{\Gamma\in \mathscr{M}\ ;\ \Gamma \cap \{\Ao\subset A\} \in \sigma(A) {\rm \ for \ any\ closed\ set\ }A\}),\]
\[\sigma(\overline{{\Bo}^c}):=\sigma(\{\Gamma\in \mathscr{M}\ ;\ \Gamma \cap \{\overline{{\Bo}^c}\subset A\} \in \sigma(A) {\rm \ for \ any\ closed\ set\ }A\}),\]
\[\sigma(\Ao\cap \overline{{\Bo}^c}):=\sigma(\Ao)\cap\sigma(\overline{{\Bo}^c}).\]
\end{definition}

\begin{remark}
The above definition is a generalization of the deterministic case. Suppose that \(\Ao=D_1\) for a deterministic closed set \(D_1.\) For any \(\Gamma \in \sigma(\Ao)\), we have \(\Gamma=\Gamma \cap \{\Ao=D_1\subset D_1\}\in \sigma(D_1)\), and so we obtain \(\sigma(\Ao)\subset \sigma(D_1).\) Conversely, for \(\Gamma \in \sigma(D_1),\) if \(D_1\subset A\), \(\Gamma \cap \{\Ao\subset A\}\) is equal to \(\Gamma\). Otherwise it is the empty set, and so it holds that \(\Gamma \cap \{\Ao\subset A\}\in \sigma(A)\) for any closed set \(A\). Hence we have \(\sigma(\Ao)=\sigma(D_1).\) If \(\overline{{\Bo}^c}=D_2\) for a deterministic closed set \(D_2\), \(\sigma(\overline{{\Bo}^c})=\sigma(D_2)\) follows by the same argument. Moreover we have \(\sigma(\Ao\cap \overline{{\Bo}^c})=\sigma(\Ao)\cap\sigma(\overline{{\Bo}^c})=\sigma(D_1)\cap\sigma({D_2})=\sigma(D_1\cap{D_2}).\)
\end{remark}

\begin{definition}
Let \(\Ao\) and \(\Bo\) be random sets satisfying \(\Bo\subset \Ao\) such that \(\Ao\) and \(\overline{{\Bo}^c}\) are stopping sets. \(\G(\cE)\) has \emph{the strong Markov property} with respect to \(\Ao, \Bo\) if \(\sigma (\Ao)\) is conditionally independent of \(\sigma (\overline{\Bo^c})\) given \(\sigma(\Ao\cap \overline{{\Bo}^c})\).
\end{definition}

The following theorem is the second main theorem of this paper and concerns the strong Markov property.
\begin{theorem}\label{mainstrong} For closed random sets \(\Ao, \Bo\) satisfying \(\Bo\subset \Ao\), suppose that \(\{\Ao\subset A,\  \overline{{\Bo}^c}\subset B \}\in \sigma(A\cap B)\) for all closed sets \(A,B\). Then \(\G(\cE)\) has the strong Markov property with respect to \(\Ao, \Bo\).
\end{theorem}
The proof of Theorem \ref{mainstrong} is given in Section \(\ref{SMarkov}\).

As far as we are concerned, this is the first work that discusses the strong Markov property for \(\G(\cE)\). Note that the strong Markov property for GFFs that appeared in \cite{Ar, Se} is a strong version of the domain Markov property (also called the spatial Markov property), and is different from the strong Markov property treated in this paper.

\section{The Markov property}\label{Markov}
In this section, we prove Theorem \ref{mainmain}. We first give some preliminaries. Let \((\cE,\cF)\) be an irreducible regular Dirichlet form on \(L^2(E;m)\). For an open set \(U\), set \(\cF^U:=\{f\in \cF: \tilde{f}=0\  \rm{q.e.\ on}\  \mathit{U^c}\}\) and \(\cF^U_e:=\{f\in \cF_e: \tilde{f}=0\  \rm{q.e.\ on}\ \mathit{U^c}\}\), where \(\tilde{f}\) be quasi-continuous version of \(f\). For a Borel set \(A\), we define the hitting time of \(A\) as \(\tau_A:=\inf \{t>0: Z_t\in A\}\) and the hitting distribution of \(A\) as \(\bH_{A}f(x):=\E_x(f(Z_{\tau_A}))\) for a Borel measurable function \(f\) and \(x\in E\), where \(Z\) is the Hunt process associated with \((\cE,\cF)\) on \(L^2(E;m)\).

By \cite[Theorem 3.4.8.]{CF}, for any \(f\in \cF_e\) and for any Borel set \(A\), \(\bH_A|\tilde{f}|\) is finite q.e. and \(\bH_A\tilde{f}\) is a quasi continuous function in \(\cF_e\) satisfying \(\cE(\bH_A\tilde{f},g)=0\) for any \(g\in \cF^{A^c}_e\).
By \cite[Theorem 2.3.3]{FOT}, for a closed set \(F\) and \(f\in \cF_e\), \(s(f)\subset F\) is equivalent to \(\cE(f,g)=0\) for any \(g\in C_c(F^c)\cap \cF\). Therefore we have \(s(\bH_F\tilde{g})\subset F\) for a closed set \(F\) and \(g\in \cF_e\).

\begin{lemma}\((\)\cite[Lemma 2.1]{FO2}\()\)\label{Flemma1}
For any \(A\subset E\) and \(f\in \cF_e\), it holds that 
\begin{equation}
\E(X_f|\sigma (\overline{A}))=X_{\bH_{\overline{A}}\tilde{f}}. \label{eq:psMar}\end{equation}
\end{lemma}
\begin{proof}
This follows from the fact that \(\cE(f-\bH_{\overline{A}}\tilde{f},g)=0\) for any \(A\) and \(f,g\in \cF_e\) with \(s(g)\subset \overline{A}\). See \cite[Lemma 2.1]{FO2} for the detailed proof.
\end{proof}

\begin{lemma}\(((2.16)\) in \cite{FO2}\()\)\label{Flemma}
\(\G(\cE)\) has the Markov property with respect to \(A\) if and only if the following holds:
\begin{equation}
\sigma (X_{\bH_{\overline{A}}\tilde{f}}: s(f)\subset \overline{A^c})\subset \sigma (\partial A).\label{eq:Markov1}\end{equation}
\end{lemma}
\begin{proof}
By \cite[Lemma 2.1]{P}, \(\G(\cE)\) has the Markov property with respect to \(A\) if and only if it holds that \(\sigma (\E(X_f|\sigma (\overline{A})): s(f)\subset \overline{A^c})\subset \sigma (\partial A)\). (\ref{eq:Markov1}) follows from this and (\ref{eq:psMar}).
\end{proof}

\begin{remark}\label{Fukurem}
For any irreducible regular Dirichlet form, \(\sigma(\overline{A})\) is conditionally independent of \(\sigma(\overline{A}^c)\) given \(\sigma(X_{\bH_{\overline{A}}\tilde{f}}: s(f)\subset \overline{A^c})\) by \(
(\ref{eq:psMar})\). This conditionally independence is called a pseudo Markov property in \cite{FO2}. A pseudo Markov property can be regarded as a domain Markov property and it is used in many literatures, for example \cite{DS},\cite{HMP},\cite{SS}, etc.
\end{remark}

To prove the Markov property with respect to \(A\), it is enough to show that \((\ref{eq:Markov1})\). By the definition of \(\bH_{\overline{A}}\), intuitively, we consider only before \(Z\) hits \(A\), and the restriction of \((\cE,\cF)\) to \(A\) is transient if \(m(A^c)>0\). So, in order to prove Theorem \ref{mainmain} when \((\cE,\cF)\) is recurrent, we use the equivalence of the transient case proved by R\"{o}ckner \cite{R}.

\begin{proof}[Proof of the equivalence between \((1),(2),(3)\) of Theorem \ref{mainmain}]
When \((\cE,\cF)\) is transient, the equivalence was proved by R\"{o}ckner \cite{R}, so we assume that \((\cE,\cF)\) is recurrent. \((3)\) implies \((2)\) clearly, and \((2)\) implies \((1)\) by \cite{FO2}.

We show that \((1)\) implies \((2)\). Suppose that \((\cE,\cF)\) is local. For a relatively compact open set \(A\subset E\), by Lemma \ref{Flemma}, it is enough to show that \(s(\bH _{\overline{A}}f)\subset \partial A\) for \(f\in \cF_e\) with \(s(f)\subset \overline{A^c}\). Fix \(f\in \cF_e\) with \(s(f)\subset \overline{A^c}\).

If \((\overline{A})^c\) is the empty set, it holds that \(\bH_{\overline{A}}f=f\ a.e.\) and \(C_c((\partial A)^c)\cap \cF = C_c(A)\cap \cF\). So, there is nothing to prove.

If \((\overline{A})^c\) is not the empty set, we take an open set \(U\) satisfying \(\overline{A}\subset U\) and \(m(U^c)>0\). Indeed, since \(E\) is a metric space, there exist non-empty open sets \(U\) and \(V\) such that \(U\cap V=\phi\) and \(\overline{A}\subset U\). Since \(m\) is full-support, it holds that \(m(V)>0\) and we have \(m(U^c)>0\).

Since \(s(f)\subset \overline{A^c}\), it holds that \(\cE(f-\bH_{U^c}f,g_1)=0\) for \(g_1\in C_c(A)\cap \cF^{U}\). Since \((\cE,\cF^U)\) on \(L^2(U;m|_U)\) is a transient Dirichlet form, by \cite{R}, it holds that 
\begin{equation}\label{eq:j1}
\cE(\bH_{\overline{A}}^U(f-\bH_{U^c}f),g)=0
\end{equation} for \(g\in C_c((\partial A)^c)\cap \cF^{U}\), where \(\bH_{\overline{A}}^Uf(x):=\E_x(f(Z_{\tau_{\overline{A}}}); \tau_{\overline{A}}<\tau_{U^c})\).

For \(g\in C_c((\partial A)^c)\cap \cF^{U}\), \(\bH_{U^c}\bH_{\overline{A}}^Ug\) vanishes on \(U^c\), so we have  
\begin{eqnarray}
\cE(\bH_{\overline{A}}^U\bH_{U^c}f,g)&=&\cE(\bH_{U^c}f,\bH_{\overline{A}}^Ug) \nonumber \\
&=&\cE(\bH_{U^c}f,\bH_{U^c}\bH_{\overline{A}}^Ug) \nonumber \\
&=&0.\label{eq:j2}
\end{eqnarray}
By \((\ref{eq:j1})\) and \((\ref{eq:j2})\), for \(g\in C_c((\partial A)^c)\cap \cF^{U}\), we have \begin{equation}\label{eq:main2}\cE(\bH_{\overline{A}}^Uf,g)=0.\end{equation}

Take \(h\in C_c((\partial A)^c)\cap \cF\). Let \(K\) be the compact set defined by \({\rm supp}(h)\cap \overline{A}\). By \cite[Exercise 1.4.1]{FOT}, since \(C_c(E)\cap \cF\) is a special standard core, there exists a non-negative function \(\psi \in C_c(E)\cap \cF\) such that \(\psi=1\) on \(K\) and \(\psi =0\) on \(A^c\). Moreover it holds that \(h\psi \in C_c(E)\cap \cF\), \(h\psi =0\) on \(U^c\) and \({\rm supp}(h\psi)\subset {\rm supp}(h)\subset (\partial A)^c\).

By setting \(g:=h\psi\) in \((\ref{eq:main2})\), we have \begin{equation}\label{eq:main3}
\cE(\bH_{\overline{A}}^Uf,h\psi)=0.
\end{equation}

Now we have \(\bH_{\overline{A}}f-\bH_{\overline{A}}^Uf\in \cF_e\) and \(\bH_{\overline{A}}f-\bH_{\overline{A}}^Uf=0\) on \(\overline{A}\), so it is in \(\cF_e^{(\overline{A})^c}\). By \cite[Theorem 3.4.9]{CF}, there are \(\cE\)-Cauchy sequence \(\{f_n\}\subset \cF^{\overline{A}^c} \) such that \(f_n\) convergent to \(\bH_{\overline{A}}f-\bH_{\overline{A}}^Uf\ a.e.\) Since \(f_n \cdot h\psi=0\), we have \(\cE(f_n,h\psi)=0\) by the locality of \((\cE,\cF)\), and \(\cE(\bH_{\overline{A}}f-\bH_{\overline{A}}^Uf,h\psi)=0\). Combining this with \((\ref{eq:main3})\), we obtain
\begin{equation}\label{eq:main4}
\cE(\bH_{\overline{A}}f,h\psi)=0
\end{equation}
Since \(h=0\) on \(\overline{A}\setminus K\) and \(1-\psi=0\) on \(K\), we have \(h(1-\psi)=0\) on \(\overline{A}\). So we obtain \begin{equation}\label{eq:main5}
\cE(\bH_{\overline{A}}f,h(1-\psi))=0.
\end{equation}

Thus, by \((\ref{eq:main4})\) and \((\ref{eq:main5})\), we have
\begin{equation*}
\cE(\bH_{\overline{A}}f,h)=0
\end{equation*}
This means \(s(\bH _{\overline{A}}f)\subset \partial A\).

Next, we prove that \((1)\) implies \((3)\). Suppose that \((\cE,\cF)\) is local. Let \(B\) be an open set. It is enough to show that \(s(\bH _{\overline{B}}f)\subset \partial B\) for \(f\in \cF_e\) with \(s(f)\subset \overline{B^c}\). We fix \(f\in \cF_e\) with \(s(f)\subset \overline{B^c}\). Take \(h \in C_c((\partial B)^c)\cap \cF\) and set \(K:={\rm supp} (h)\cap \overline{B}\), which is compact. Since \(E\) is locally compact separable, there is relatively compact open set \(A\) satisfying \(K\subset A \subset \overline{A}\subset B\). Since \(C_c(E)\cap \cF\) is a special standard core, there exists a non-negative function \( \varphi \in C_c(E)\cap \cF\) such that \(\varphi=1\) on \(K\), \(\varphi=0\) on \(A^c\) and \(0\leq \varphi \leq 1\). As \(s(f)\subset \overline{B^c}\subset \overline{A^c}\), by the Markov property with respect to relatively compact open set \(A\), we have
\begin{equation}\label{eq:main6}
\cE(\bH_{\overline{A}}f,g)=0
\end{equation}  
for \(g\in C_c((\partial A)^c)\cap \cF\). Since \(h\varphi \in C_c(E)\cap \cF\) and \(h\varphi=0\) on \(A^c=\overline{A^c}\), we have \(h\varphi \in C_c((\partial A)^c)\cap \cF\). By \((\ref{eq:main6})\), it holds that
\begin{equation}\label{eq:main7}
\cE(\bH_{\overline{A}}f,h\varphi)=0.
\end{equation}
Since \(\bH_{\overline{B}}f-\bH_{\overline{A}}f=0\) on \(\overline{A}\), by the local property of \((\cE, \cF)\) and \((\ref{eq:main7})\), we have 
\begin{equation}\label{eq:main8}
\cE(\bH_{\overline{B}}f,h\varphi)=\cE(\bH_{\overline{A}}f,h\varphi)=0.
\end{equation}
Since \(h(1-\varphi)\in \cF^{\overline{B}^c}\), it holds that
\begin{equation}\label{eq:main9}
\cE(\bH_{\overline{B}}f,h(1-\varphi))=0.
\end{equation}
By \((\ref{eq:main8})\) and \((\ref{eq:main9})\), we have \(\cE(\bH_{\overline{B}}f,h)=0\). Thus \((3)\) is proved.
\end{proof}

In order to prove the equivalence of \((1)\) and \((4)\) in Theorem \ref{mainmain}, we prove the following theorem.
\begin{theorem}\label{main2}
\((1),(2),(3)\) of Theorem \ref{mainmain} are equivalent to the following:\\
\((3^*)\) \(\sigma(\overline{A})\) is conditionally independent of \(\sigma (\overline{B^c})\) given \(\sigma (\overline{A}\cap\overline{B^c})\) for any open or closed set \(A\) and any open or closed set \(B\) satisfying \(B\subset A.\)
\end{theorem}

\begin{remark}
Sometimes the above type of the Markov property is adopted as the definition of the Markov property for set-indexed processes. See \cite{K} and \cite{E2} for examples.
\end{remark}

In order to prove Theorem \ref{main2}, we use the following lemma.
\begin{lemma}\label{lemofmain2}
For any sets \(A\) and \(B\) satisfying \(B\subset A\), it holds that \[\sigma(\overline{A^c}) \vee \sigma(\overline{A}\cap \overline{B^c})=\sigma(\overline{B^c}).\]
\end{lemma}
\begin{proof}
\(\overline{A^c}\subset \overline{B^c}\) and \(\overline{A}\cap\overline{B^c}\subset \overline{B^c}\) yield that \(\sigma(\overline{A^c}) \vee \sigma(\overline{A}\cap \overline{B^c})\subset \sigma(\overline{B^c}).\) We will show the other direction. Fix \(f\in \cF_e\) satisfying \(s(f)\subset \overline{B^c}\).

For any \(g\in C_c((\overline{A}\cap \overline{B^c})^c)\), there exist \(g_1, g_2 \in \cF\cap C_c(E)\) such that \(g=g_1+g_2\), \({\rm supp}(g_1)\subset \overline{A^c}\) and \({\rm supp}(g_2)\subset (\overline{B^c})^c\) by the similar way as the proof of Theorem \(\ref{mainmain}\). In fact, for a relatively compact open set \(U\) satisfying \(K:={\rm supp}(g)\cap \overline{\overline{B^c}^c} \subset U \subset \overline{U}\subset \overline{A}\), there exists \(\varphi \in \cF \cap C_c(E)\) with \(\varphi|_K=1, \varphi|_{(\overline{A})^c}=0\) and \(0\leq \varphi \leq 1\), and we set \(g_1:=g(1-\varphi)\) and \(g_2:=g\varphi\). Then we have
\begin{eqnarray}
\cE(f-\bH_{\overline{A^c}}f,g)&=&\cE(f-\bH_{\overline{A^c}}f,g_1)+\cE(f-\bH_{\overline{A^c}}f,g_2)\nonumber\\
&=&\cE(f-\bH_{\overline{A^c}}f,g_1)+\cE(f,g_2)-\cE(\bH_{\overline{A^c}}f,g_2) \nonumber\\
&=&0. \nonumber
\end{eqnarray}
In the last equality, we used the local property of \((\cE,\cF)\), \(s(f)\subset \overline{B^c}\) and \(\bH_{\overline{A^c}}g_2=0\). Therefore we have \(s(f-\bH_{\overline{A^c}}f)\subset \overline{A}\cap\overline{B^c}.\)

For any \(h\in \cF\), it holds that \(\mathbb{E}((X_f-X_{f-\bH_{\overline{A^c}}f}-X_{\bH_{\overline{A^c}}f})X_h)=0\), so we have \(X_f-X_{\bH_{\overline{A^c}}f}=X_{f-\bH_{\overline{A^c}}f}\) and \((X_f)^{-1}(I) \in \sigma (X_{f-\bH_{\overline{A^c}}f},X_{\bH_{\overline{A^c}}f})\) for a Borel set \(I\subset \mathbb{R}\). Since \(s(\bH_{\overline{A^c}}f)\subset \overline{A^c}\) and \(s(f-\bH_{\overline{A^c}}f)\subset \overline{A}\cap\overline{B^c}\), it holds that  \((X_f)^{-1}(I) \in \sigma(\overline{A^c})\vee \sigma(\overline{A}\cap\overline{B^c}).\) Thus we have \(\sigma(\overline{B^c})\subset \sigma(\overline{A^c}) \vee \sigma(\overline{A}\cap \overline{B^c}).\)
\end{proof}

\begin{proof}[Proof of the Theorem \ref{main2}]
The equivalence of \((1),(2)\) and \((3)\) is already proved. \((3^*)\) cleary implies \((3)\) by setting \(A=B.\) We now show that \((3)\) implies \((3^*).\) For open or closed sets \(A\) and \(B\) with \(B\subset A\), it is enough to show that \begin{equation}\label{eq:main21}\mathbb{P}(\Gamma|\sigma(\overline{A}))=\mathbb{P}(\Gamma|\sigma(\overline{A}\cap\overline{B^c}))\end{equation}
 for any \(\Gamma \in \sigma(\overline{B^c}).\) By Lemma \ref{lemofmain2}, we may assume that \(\Gamma \in \sigma (\overline{A^c})\) or \(\Gamma \in \sigma(\overline{A}\cap \overline{B^c})\).

For \(\Gamma \in \sigma (\overline{A^c})\), since \(A\) or \(A^c\) is an open set, by the Markov property with respect to an open set, we have \(\mathbb{P}(\Gamma|\sigma(\overline{A}))=\mathbb{P}(\Gamma|\sigma(\partial A))\), and (\ref{eq:main21}) follows from the fact \(\sigma(\partial A)\subset \sigma(\overline{A}\cap\overline{B^c})\subset \sigma(\overline{A})\).

For \(\Gamma \in \sigma(\overline{A}\cap \overline{B^c})\), we have \(\Gamma \in \sigma(\overline{A})\), and \(\mathbb{P}(\Gamma|\sigma(\overline{A}))=\Gamma=\mathbb{P}(\Gamma|\sigma(\overline{A}\cap\overline{B^c}))\). Thus the proof is completed.
\end{proof}

\begin{proof}[Proof of the equivalence between \((3)\) and \((4)\) of Theorem \ref{mainmain}]
It is clear that \((4)\) implies to \((3)\). Suppose that \((3)\) holds. Then \((3^*)\) holds by Theorem \ref{main2}. For any subset \(A\subset E\), since \(\overline{A^c}\supset A^c \supset (\overline{A})^c\) , we have \(\overline{A^c}^c \subset \overline{A}\). By \((3^*)\), \(\sigma(\overline{A})\) is conditionally independent of \(\sigma(\overline{A^c})\) given \(\sigma(\partial A)\). Thus the proof is completed.
\end{proof}

\section{The Strong Markov property}\label{SMarkov}
In this section, we prove Theorem \ref{mainstrong}. Throughout this section, we assume that \((\cE,\cF)\) is an irreducible regular local Dirichlet form on \(L^2(E;m)\). By Theorem \ref{mainmain}, \(\G(\cE)\) has the Markov property.
 
In Section \ref{intro}, we defined the filtrations for stopping sets and the strong Markov property. These can be viewed as analogues to those for a Markov process. Indeed, for a Markov process \(Z=\{Z_t\}_{t\geq 0}\), let \(\{\mathscr{F}_t\}_{t\geq 0}\) be a filtration to which \(Z\) is adapted, and  \(\tau\) be a stopping time. \(Z\) has the strong Markov property if \(\{Z_{\tau+s}\}_{s\geq 0}\) is conditionally independent of \(\mathscr{F}_{\tau}\) given \(\sigma(Z_{\tau})\), where \(\mathscr{F}_{\tau}:=\sigma(\{\Gamma\ ;\ \Gamma \cap \{\tau\leq t\}\in \mathscr{F}_t\ {\rm for\ all\ } t\})\). By considering \(\sigma(\overline{{\Ao}^c})\) corresponds to the \(\sigma\)-field generated by \(\{Z_{\tau+s}\}_{s\geq 0}\), we can regard the definitions of filtrations for stopping sets as the correspondences to those for a process.

\begin{remark}\label{remev}
For set-indexed processes, there are other definitions of \(\sigma\)-fields for random sets. For example, Evstigneev \cite{E1,E2} and Kinateder \cite{K} defined the correspondence to \(\sigma(\Ao\cap \overline{{\Bo}^c})\) as that includes the \(\sigma\)-field generated by \(\Ao\) and \(\Bo\).
\end{remark}

\begin{remark}\label{remstop}
Under the assumption of Theorem \ref{mainstrong}, \(\Ao\) and \(\overline{\Bo^c}\) are stopping random sets. Indeed, it holds that 
\begin{eqnarray*}
\{\Ao\subset A\}=\{\Ao\subset A,\  \overline{{\Bo}^c}\subset E \}\in \sigma(A\cap E)=\sigma(A),\\
\{\overline{\Bo^c}\subset B\}=\{\Ao\subset E,\  \overline{{\Bo}^c}\subset B \}\in \sigma(E\cap B)=\sigma(B)\ 
\end{eqnarray*} for any closed sets \(A\) and \(B\).
\end{remark}

To prove Theorem \ref{mainstrong}, we approximate \(\sigma\)-fields of random sets by those of discrete random sets. We need some properties for these approximating \(\sigma\)-fields \(\sigma(\Ao^n), \sigma(\overline{\Bo^c}^n)\) such as the monotonicity for \(n\) and the conditional independence of \(\sigma(\Ao^n)\) and \(\sigma(\overline{\Bo^c}^n)\) given \(\sigma(\Ao^n\cap \overline{\Bo^c}^n)\). This is because the strong Markov property follows from the strong Markov properties for approximating sequences by taking the limit as \(n\) goes to infinity. So we first define discrete random sets \(\Ao^n, \overline{\Bo^c}^n\) to approximate \(\Ao, \overline{\Bo^c}\), respectively.

Since \(E\) is a separable metric space, there exists a countable open basis \(\{U_i\}_{i=1}^{\infty}\). For a closed set \(F\), we set \(F^n:=\bigcap_{i\in \Lambda^n_F} {U_i}^c\) where \(\Lambda^n_F:=\{1\leq i\leq n\ ;\ F\subset {U_i}^c\}\). Then it holds that \(F^{n}\supset F^{n+1}\) and \(F=\bigcap_{i=1}^{\infty}F^n.\)

\begin{lemma}\label{splitn}
Under the assumption of Theorem \ref{mainstrong}, it holds that \(\{\Ao^n\subset A,\  {\overline{\Bo^c}}^n\subset B \}\in \sigma(A\cap B)\) for all closed sets \(A,B\) and \(n\geq 1\).
\end{lemma}
\begin{proof}
If \(A\) or \(B\) is not the form of \(\bigcap _{i\in \Lambda} U_i^c\) for any \(\Lambda \subset \{1,\cdots,n\}\), then \(\{\Ao^n=A, {\overline{\Bo^c}}^n=B\}\) is the empty set and so is in \( \sigma(A\cap B)\). If \(A=\bigcap _{i\in \Lambda} U_i^c\) and \(\overline{\Bo^c}=\bigcap _{i\in \Sigma} U_i^c\) for some \(\Lambda, \Sigma \subset \{1,\cdots, n\}\), then we have
\begin{eqnarray*}
\{\!\Ao^n\!= \! A,\  \!{\overline{\Bo^c}}^n\!\!= B \}\!=\!\{\Ao\subset \!A,\  \overline{\Bo^c}\subset B\}\!\cap\! \bigcap_{j\not \in \Lambda}\{\Ao\not \subset A\cap U_j^c\}\! \cap \!\bigcap_{j\not \in \Sigma}\{\overline{\Bo^c}\not \subset B\cap U_j^c\}\hspace{8mm}\\
=\!\{\Ao\subset \!A,\  \overline{\Bo^c}\subset B\}\!\cap\! \bigcap_{j\not \in \Lambda} \left(\{\Ao\subset \!A,\  \overline{\Bo^c}\subset B\}\cap\{\Ao \subset A\cap U_j^c\}\right)^c \hspace{25mm}\\
\cap \!\bigcap_{j\not \in \Sigma}\left(\{\Ao\subset \!A,\  \overline{\Bo^c}\subset B\}\cap\{\overline{\Bo^c} \subset B\cap U_j^c\}\right)^c\hspace{15mm}\\
\!=\!\{\!\Ao\!\subset \!A,\!\  \overline{\Bo^c}\!\subset \!B\}\setminus \left(\bigcup_{j\not \in \Lambda}\{\!\Ao\!\subset \!A\!\cap U_j^c,\!\  \overline{\Bo^c}\!\subset B\!\}\cup\bigcup_{j\not \in \Sigma}\{\!\Ao\!\subset \!A,\!\  \overline{\Bo^c}\!\subset \! B\!\cap \! U_j^c\!\}\right).\hspace{6mm}\end{eqnarray*}
Hence we have
 \[\{\Ao^n\!= A,\  {\overline{\Bo^c}}^n\!\!= B \}
\in \sigma(A\cap B) \vee \bigvee_{j\not \in \Lambda}\sigma(A\cap U_j^c \cap B)\vee \bigvee_{j\not \in \Sigma}\sigma(A\cap B\cap U_j^c)\subset \sigma(A\cap B).\]
Moreover we have \(\{\Ao^n\!\subset A,\  {\overline{\Bo^c}}^n\!\!\subset B \}=\bigcup _{\tilde{A}\subset A,\ \tilde{B}\subset B}\{\Ao^n\!= \tilde{A},\  {\overline{\Bo^c}}^n\!\!= \tilde{B} \} \in \sigma(A\cap B)\) for any closed sets \(A,B.\)
\end{proof}
By this lemma and Remark \ref{remstop}, the following holds.
\begin{corollary}\label{stoppingn}
Under the assumption of Theorem \ref{mainstrong}, \(\Ao^n\) and \(\overline{\Bo^c}\) are stopping sets.
\end{corollary}
\begin{remark}\label{stoppingnn}
The assumption of Theorem \ref{mainstrong} is not needed for \(\Ao^n\) to be a stopping set. By the similar proof as one of Lemma \ref{splitn}, we know it is sufficient to assume that \(\Ao\) is a stopping set. The same statement holds for \(\overline{\Bo^c}^n\).
\end{remark}

Next, we obtain the following lemma about the monotonicity.
\begin{lemma}\label{lemmadec}
If \(\Ao\) is a stopping set, then \(\sigma(\Ao^n)\) is decreasing to \(\sigma(\Ao).\)
\end{lemma}
\begin{corollary}\label{cordec}
Under the assumption of Theorem \ref{mainstrong}, \(\sigma(\Ao^n)\) is decreasing to \(\sigma(\Ao)\), and \(\sigma(\overline{\Bo^c}^n)\) is decreasing to \(\sigma(\overline{\Bo^c}).\)
\end{corollary}

\begin{proof}[Proof of Lemma \ref{lemmadec}]
For any \(n\) and \(\Gamma \in \sigma(\Ao^{n+1})\), since \(\Ao^{n+1}\subset \Ao^n\) and \(\Ao^n\) is a stopping set,  we have \(\Gamma \cap \{\Ao^n\subset A\}=(\Gamma \cap \{\Ao^{n+1}\subset A\})\cap \{\Ao^n \subset A\}\in \sigma(A).\) So we have \(\sigma(\Ao^{n+1})\subset\sigma(\Ao^n).\) By the same way, we have \(\sigma(\Ao)\subset\sigma(\Ao^n)\) for any \(n,\) and we have \(\sigma(\Ao)\subset \bigcap_{n=1}^{\infty} \sigma(\Ao^n).\) Conversely, for \(\Gamma \in \bigcap_{n=1}^{\infty}\sigma(\Ao^n)\) and any closed set \(A\), we have \(\Gamma \cap \{\Ao \subset A\}=\Gamma \cap \{\Ao^n \subset A^n\}\in \sigma(A^n)\) and \(\Gamma \cap \{\Ao \subset A\} \in \bigcap_{n=1}^{\infty}\sigma(A^n)\). By the similar proof of \cite[Theorem 5.6.]{R} with \(\mu^A\) replaced by \(\bH_{A}f\) for \(f\in \cF_e\), it holds that \(\sigma(A)=\bigcap_{n=1}^{\infty}\sigma(A^n)\) and we have \(\Gamma \cap \{\Ao \subset A\}\in \sigma(A)\).
\end{proof}

In the following lemmas, we show that other types of definitions about \(\sigma\)-fields are the same as what we defined.
\begin{lemma}\label{lemmadisc}
If \(\Ao\) is a discrete stopping set, then it holds that 
\begin{eqnarray*}
\sigma(\Ao)&=&\sigma(\Gamma\ ;\ \Gamma \cap \{\Ao=A\}\in \sigma(A) {\rm \ for\ any\ closed\ set\ }A)\\
&=&\bigvee_{A:{\rm\ closed}}\sigma(\Gamma \cap \{\Ao=A\}\ ;\ \Gamma \in \sigma (A)).
\end{eqnarray*}
\end{lemma}
\begin{proof}
Set \(\Ao=\sum_{i=1}^{\infty} {\bf 1}_{\Gamma_i}A_i\) for \(\Gamma_i \subset \mathscr{M}\) which are pairwise disjoint, and closed sets \(A_i\subset E\). We show the first equality. For \(\Gamma \in \sigma(\Ao)\) and any closed set \(A\), we have
\[\Gamma \cap \{\Ao= A\}=(\Gamma \cap \{\Ao\subset A\})\setminus \bigcup_{A_i \subsetneq A}(\Gamma \cap \{\Ao\subset A_i\})\in \sigma(A).\] 
Conversely, take \(\Gamma\) satisfying \(\Gamma \cap \{\Ao=A\}\in \sigma(A)\) for any closed set \(A\). For any closed set \(A,\) we have 
\[\Gamma \cap \{\Ao \subset A\}=\bigcup_{A_i\subset A}(\Gamma \cap \{\Ao =A_i\})\in \sigma(A).\] So the first equality holds.\\
Next we show the second equality. For \(\Gamma \in \sigma(\Ao),\) we have \[\Gamma=\bigcup_{i=1}^{\infty}\left((\Gamma\cap \{\Ao=A_i\})\cap\{\Ao=A_i\}\right),\] and so we have \[\sigma(\Ao)\subset \bigvee_{A:{\rm\ closed}}\sigma(\Gamma \cap \{\Ao=A\}\ ;\ \Gamma \in \sigma (A)).\]
 Conversely, fix any closed set \(A\) and \(\Gamma \in \sigma(A).\) If \(A=A_i=B\) for some \(i,\) \(\Gamma \cap \{\Ao=A\}\cap \{\Ao=B\}\) is equal to \(\Gamma \cap \{\Ao=B\}\in \sigma(B)\). Otherwise it is the empty set. Hence we have \(\sigma(\Ao)\supset \bigvee_{A:{\rm\ closed}}\sigma(\Gamma \cap \{\Ao=A\}\ ;\ \Gamma \in \sigma (A))\). Thus the second equality holds. 
\end{proof}

\begin{corollary}\label{cordisc2}
Define \begin{eqnarray*}
\tilde{\sigma}(\Ao^n):=\bigvee_{A:{\rm\ closed}}\sigma(\Gamma \cap \{\Ao^n=A\}\ ;\ \Gamma \in \sigma (A)),\\
\tilde{\sigma}(\overline{\Bo^c}^n):=\bigvee_{A:{\rm\ closed}}\sigma(\Gamma \cap \{\overline{\Bo^c}^n=A\}\ ;\ \Gamma \in \sigma (A)).
\end{eqnarray*}
Under the assumption of Theorem \ref{mainstrong}, it holds that \(\sigma(\Ao^n)= \tilde{\sigma}(\Ao^n) \) and \(\sigma(\overline{\Bo^c}^n)=\tilde{\sigma}(\overline{\Bo^c}^n)\).
\end{corollary}

\begin{remark}
For the case of set-indexed processes, Evstigneev, Ovseevi\v{c} \cite{EO} and Balan \cite{B} used the definitions of the forms \(\tilde{\sigma}(\Ao^n)\) and \(\tilde{\sigma}(\overline{\Bo^c}^n)\).
\end{remark}

In Lemma \ref{lemmadec} and \ref{lemmadisc}, it is only assumed that \(\Ao\) and \(\overline{\Bo^c}\) are stopping sets. As in Remark \ref{stoppingnn}, if \(\Ao\) and \(\overline{\Bo^c}\) are simply stopping sets, \(\Ao^n\) and \(\overline{\Bo^c}^n\) are also stopping sets, so Corollary \ref{cordec}, \ref{cordisc2} also hold. On the other hand, the assumption \(\{\Ao\subset A,\  \overline{{\Bo}^c}\subset B \}\in \sigma(A\cap B)\) in Theorem \ref{mainstrong} is needed to prove the following lemma.
\begin{lemma}\label{lemmaAB}
Under the assumption of Theorem \ref{mainstrong}, it holds that
\[\tilde{\sigma}(\Ao^n\cap\overline{\Bo^c}^n) \subset \sigma(\Ao^n\cap\overline{\Bo^c}^n),\]
where \[\tilde{\sigma}(\Ao^n\cap\overline{\Bo^c}^n):= \bigvee_{A,B:{\rm closed}}\sigma\left(\Gamma\cap \{\Ao^n=A, \overline{\Bo^c}^n=B\}\ ;\ \Gamma\in \sigma(A\cap B)\right).\]
\end{lemma}

\begin{proof}
For any closed sets \(A,B\) and \(\Gamma \in \sigma(A\cap B),\) by Lemma \ref{splitn}, we have \(\Gamma\cap \{\Ao^n=A, \overline{\Bo^c}^n=B\}\in \sigma(A\cap B).\) If \(A=\tilde{A}\), we have 
\[\Gamma\cap \{\Ao^n=A, \overline{\Bo^c}^n=B\} \cap \{\Ao^n=\tilde{A}\}=\Gamma\cap \{\Ao^n=\tilde{A}, \overline{\Bo^c}^n=B\} \in \sigma(\tilde{A}).\] 
Otherwise it is the empty set. Hence we have \(\tilde{\sigma}(\Ao^n\cap\overline{\Bo^c}^n)\subset \sigma(\Ao^n)\), and by the same way as above, it hods that \(\tilde{\sigma}(\Ao^n\cap\overline{\Bo^c}^n) \subset \sigma(\overline{\Bo^c}^n)\).
\end{proof}

\begin{remark}
In the set-indexed processes, Evstigneev \cite{E1,E2} and Kinateder \cite{K} do not assumed the correspondence to \(\{\Ao^n\subset A,\  {\overline{\Bo^c}}^n\subset B \}\in \sigma(A\cap B)\).
However, as in Remark \ref{remev}, they defined the correspondence to \(\sigma(\Ao\cap \overline{{\Bo}^c})\) to include the \(\sigma\)-field generated by \(\Ao\) and \(\Bo\). Because of this definition, the correspondence to \(\tilde{\sigma}(\Ao^n\cap\overline{\Bo^c}^n)\) is decreasing for \(n\) under the only assumption that \(\Ao\) and \(\overline{\Bo^c}\) are stopping sets. So, by taking limit, the strong Markov property follows from the strong Markov properties for approximating sequences.
\end{remark}

\begin{proposition}\label{indpendentn}
Under the assumption of Theorem \ref{mainstrong}, \(\sigma(\Ao^n)\) is conditionally independent of \(\sigma(\overline{\Bo^c}^n)\) given \(\sigma(\Ao^n\cap\overline{\Bo^c}^n).\)
\end{proposition}
\begin{proof}
Set \(\Ao^n=\sum_{i=1}^\infty {\bf 1}_{\Gamma_i}A_i\) and \(\overline{\Bo^c}^n=\sum_{j=1}^\infty {\bf 1}_{\Sigma_j}B_j\) for \(\Gamma_i, \Sigma_j \in \mathscr{M}\) and closed sets \(\{A_i\}, \{B_j\}\) satisfying \(A_i\not =A_j, B_i\not = B_j\) for \(i\not =j\).

For \(\Gamma \in \sigma(B_k),\) we have
\begin{eqnarray}
\mathbb{P}(\Gamma \cap \{\overline{\Bo^c}^n\!=\!B_k\}\ \! |\!\ \sigma(\Ao^n))=\sum_{i=1}^{\infty}\mathbb{P}(\Gamma \cap \{\overline{\Bo^c}^n\!=\!B_k\}\ \! |\!\ \sigma(A_i)\! \vee  \!\{\Ao^n\! = \! A_i\}){\bf 1}_{\{\Ao^n=A_i\}}\nonumber \\
=\sum_{i=1}^{\infty}\mathbb{P}(\Gamma \cap \{\overline{\Bo^c}^n=B_k\}\ |\ \sigma(A_i)){\bf 1}_{\{\Ao^n=A_i\}}\hspace{16mm}\nonumber \\
=\sum_{i,j=1}^{\infty}\mathbb{P}(\Gamma \cap \{\overline{\Bo^c}^n=B_k\}\ |\ \sigma(A_i)){\bf 1}_{\{\Ao^n=A_i, \overline{\Bo^c}^n=B_j\}}\hspace{2mm}\nonumber \label{eq:indn1}
\end{eqnarray}
Since \(\{\Ao^n=A_i, \overline{\Bo^c}^n=B_j\}\in \sigma(A_i\cap B_j)\subset \sigma(A_i)\), we have 
\begin{equation*}\label{eq:indn2}
\mathbb{P}(\Gamma \cap \{\overline{\Bo^c}^n=B_k\}\ |\ \sigma(\Ao^n))=\sum_{i=1}^{\infty}\mathbb{P}(\Gamma \cap \{\overline{\Bo^c}^n=B_k\}\ |\ \sigma(A_i)){\bf 1}_{\{\Ao^n=A_i, \overline{\Bo^c}^n=B_k\}}.
\end{equation*}

On \(\{\Ao^n=A_i, \overline{\Bo^c}^n=B_k\},\) we have \(B_k^c = {\overline{\Bo^c}^n}^c \subset {\overline{\Bo^c}}^c \subset \Bo \subset \Ao \subset \Ao^n=A_i\). Hence \(\sigma(A_i)\) is conditionally independent of \(\sigma(\overline{{B_k^c}^c})=\sigma(B_k)\) given \(\sigma(A_i\cap B_k)\) by the Markov property. Since \(\Gamma\cap \{\overline{\Bo^c}^n=B_k\} \in \sigma(B_k),\) it holds that
\begin{eqnarray*}
\mathbb{P}(\Gamma \! \cap \!\{\overline{\Bo^c}^n\! =\!B_k\}\ \! |\!\ \sigma(\Ao^n))=\sum_{i=1}^{\infty}\mathbb{P}(\Gamma \cap \{\overline{\Bo^c}^n\! =\!B_k\}\ \! |\!\ \sigma(A_i\cap B_k)){\bf 1}_{\{\Ao^n=A_i, \overline{\Bo^c}^n=B_k\}}\\
=\sum_{i,j=1}^{\infty}\mathbb{P}(\Gamma \cap \{\overline{\Bo^c}^n=B_k\}\ |\ \sigma(A_i\cap B_k)){\bf 1}_{\{\Ao^n=A_i, \overline{\Bo^c}^n=B_k, \overline{\Bo^c}^n=B_j\}}\\
=\sum_{i,j=1}^{\infty}\mathbb{P}(\Gamma \cap \{\overline{\Bo^c}^n=B_k\}\ |\ \sigma(A_i\cap B_j)){\bf 1}_{\{\Ao^n=A_i, \overline{\Bo^c}^n=B_j\}}\hspace{12.5mm}\\
=\mathbb{P}\left( \Gamma\cap \{\overline{\Bo^c}^n=B_k\}\ |\ \tilde{\sigma}(\Ao^n\cap \overline{\Bo^c}^n) \right).\hspace{39.5mm}
\end{eqnarray*}
where \(\tilde{\sigma}(\Ao^n\cap \overline{\Bo^c}^n):=\bigvee_{A,B:{\rm closed}}\sigma(\Gamma\cap \{\Ao^n=A, \overline{\Bo^c}^n=B\}\ ;\ \Gamma\in \sigma(A\cap B))\). By Lemma \ref{lemmaAB}, we have
\[\mathbb{P}(\Gamma \cap \{\overline{\Bo^c}^n=B_k\}\ |\ \sigma(\Ao^n))=\mathbb{P}( \Gamma\cap \{\overline{\Bo^c}^n=B_k\}\ |\ \sigma(\Ao^n\cap \overline{\Bo^c}^n)).\]
Combining this with Corollary \ref{cordisc2}, the proof is completed.
\end{proof}

\begin{proof}[Proof of Theorem \ref{mainstrong}]
It follows from Proposition \ref{indpendentn} and Lemma \ref{cordec}.
\end{proof}

\section{Examples}
In this section, we give some examples.

\begin{example}
Let us consider the closed ball \(E:=\{x\in \mathbb{R}^2\ ;\ |x|\leq 1\}\) and the Lebesgue measure \(m\) on \(E\).
We define the irreducible regular Dirichlet form \((\cE,\cF)\) on \(L^2(E;m)\) as 
\begin{eqnarray*}
\left\{ \begin{split} \cF:=H^1(E):= \left\{f\in L^2(E;m)\ ;\ \nabla f(x)\in L^2(E;m) \right\},\\
\cE(f,g):=\frac{1}{2}\int_{E} \nabla f(x)\cdot \nabla g(x)dm(x) \ \ {\rm for\ } f,g\in \cF, \hspace{8mm} \end{split}\right.
\end{eqnarray*}
and \(\G(\cE)\) be the Gaussian field associated with \((\cE,\cF)\). We may treat \(\cF\) as the family of quasi-continuous versions of the functions belonging to \(H^1(E)\).
We can identify \(\partial E\) with \([0,2\pi)\) by considering \(x=(\cos (\theta), \sin (\theta)) \in \partial E\). Let \(\check{m}\) be the uniform measure on \(\partial E\) and set \(\bH f:=\bH_{\partial E} f\) for \(f\in \cF.\) We define the irreducible regular Dirichlet form \((\check{\cE},\check{\cF})\) on \(L^2(\partial E; \check{m})\) as
\begin{eqnarray*}
\left \{ \begin{split}\check{\cF}:=\cF|_{\partial E} \cap L^2(\partial E;\check{m}), \hspace{12mm} \\ 
\check{\cE}(f,g):=\cE(\bH f,\bH g)\  {\rm for}\ f,g\in \check{\cF}. \end{split}\right.
\end{eqnarray*}
Then, by \cite[Section 5.3]{CF}, it holds that 

\begin{eqnarray*}
\left\{ \begin{split} \check{\cF}= \left\{\varphi \in L^2(\partial E;\check{m})\ ;\ \check{\cE}(\varphi,\varphi)<\infty \right\}, \hspace{30mm}\\
\check{\cE}(\varphi,\varphi)=\frac{1}{2}\int_0^{2\pi}\int_0^{2\pi} \frac{(\varphi(x)-\varphi(y))^2}{4\pi (1-\cos (x-y))} d\check{m}(x)d\check{m}(y) \ \ {\rm for\ } \varphi \in \check{\cF}. \end{split}\right.
\end{eqnarray*}
Remark that \((\cE,\cF)\) associates with the reflecting Brownian motion on \(E\) and \((\check{\cE}, \check{\cF})\) associates with the time-changed process of the reflecting Brownian motion by \(\check{m}\), which is the symmetric \(1\)-stable process on \(\partial E\).
We set \(\check{\G}(\cE):=\{X_{\bH f}\in \G(\cE)\ ;\ f\in \cF\}\) and \(\check{\sigma} (A)\) be the \(\sigma\)-field replaced \((\cE,\cF)\) by \((\check{\cE},\check{\cF})\) in the definition of \(\sigma(A)\).

Let \(A:=\{x=(x_1,x_2)\in E\ ;\ x_2\geq 0\}\). Since \((\cE,\cF)\) is local, \(\sigma(A)\) is conditionally independent of \(\sigma(\overline{A^c})\) given \(\sigma(\partial A)\) by Theorem \ref{mainmain}. However, \(\check{\sigma} (A\cap \partial E)\) is not conditionally independent of \(\check{\sigma} (\overline{A^c}\cap \partial E)\) given \(\check{\sigma}(\{0,\pi\})\) by checking that (\ref{eq:Markov1}) fails. Since \((\check{\cE},\check{\cF})\) is not local, this result is consistent to Theorem \ref{mainmain}. In this case, it holds that \(\check{\sigma} (A\cap \partial E) \subset \sigma(A)\) and \(\check{\sigma} (\overline{A^c}\cap \partial E)\subset \sigma(\overline{A^c})\), but \(\check{\sigma}(\{0,\pi\})\) is too small amount of information to make them conditionally independent.

\end{example}

\begin{example}
Let \(M\) be a continuous semimartingale on a filtered probability space \((\Omega, \mathscr{M}, \{\mathscr{F}_t\}_{t\geq 0} , \mathbb{P})\), which is the sum of a continuous local martingale and a c\`{a}dl\`{a}g adapted process with locally bounded variation. Assume that the quadratic variation \(\langle M \rangle_{\cdot}\) of \(M\) is deterministic and absolutely continuous with respect to the Lebesgue measure. We write \(d\langle M \rangle _s=h(s)ds\) for some positive function \(h\). We set
\begin{eqnarray*}
\left\{ \begin{split} \cF:=\{f:\mathbb{R}\to \mathbb{R}\ ;\ \mathbb{E}(\int_0^{\infty}|f(s)|^2d\langle M \rangle_s)<\infty\},\hspace{30mm}\\
 \cE(f,g):=\mathbb{E}(\int_0^{\infty}f(s)g(s)d\langle M \rangle_s)=\int _0^{\infty} f(s)g(s)h(s)ds\ {\rm for}\ f,g\in \cF, \end{split}\right. \end{eqnarray*}
  and \(X_f\) denotes \(\int_0^{\infty} f(s)d M _s\). Then \((\cE,\cF)\) is an irreducible local regular Dirichlet form on \(L^2([0,\infty))\). Fix \(t>0\), then it holds that \(\sigma(\{t\})=\{\Omega, \phi\}\vee \sigma(\mathscr{N})\). Indeed, for \(f\) satisfying \(s(f)=\{t\}\), we have \(\mathbb{E}(\int_0^{\infty}f(s)g(s)d\langle M \rangle_s)=0\) for \(g\in C_c(\{t\}^c)\). By the fundamental lemma of calculus of variations, we have \(f(s)h(s)=0\) for almost every \(s\). Since \(h\) is positive, we have \(f(s)=0\).

By Theorem \ref{mainmain}, \(\sigma([0,t])\) is conditionally independent of \(\sigma([t, \infty))\) given \(\sigma(\{t\})=\{\Omega, \phi\}\vee \sigma(\mathscr{N})\). Since \(\{\Omega, \phi\}\vee \sigma(\mathscr{N})\) is independent of any set, this conditional independence is independence. Combining this with the fact that \(M_u=\int {\bf 1}_{[0,u]}(s)dM_s \in \sigma([0,t])\) for \(u\leq t\) and \(M_{t+r}-M_t=\int {\bf 1}_{[t,t+r]}(s)dM_s \in \sigma([t,\infty))\) for \(r\geq 0\), \(M\) has independent increments.

\end{example}

\begin{example}\label{exampleBM} In this example, we show that the Markov property for Gaussian fields yields that for some processes induced by the field and Dirac measures.
Let \(E\) be the half line \((0,\infty)\), \(m\) be the positive Radon measure with \({\rm supp}(m)=E\). Let \((\cE,\cF)\) be an irreducible transient local regular Dirichlet form on \(L^2(E;m)\) and \(\G(\cE)\) be a Gaussian field associated with \((\cE,\cF).\) Let \(Z\) be the Hunt process associated with \((\cE,\cF)\) on \(L^2(E;m)\) and suppose \(Z\) admits no killing inside.
For any \(t>0,\) we write \(\delta_t\) the Dirac measure at \(t\). Assume that, for almost all \(t>0\), \(\delta_t\) is a measure of finite energy integral, which is the condition that there exists \(C>0\) such that \(\int_E |g(x)|d\delta_t(x)\leq C\sqrt{\cE(g,g)}\) for any \(g\in \cF \cap C_c(E).\)  By \cite[p.87] {CF}, there exists \(U\delta_t \in \cF\) such that 
\begin{equation}\cE(U\delta_t, g)=\int_E \tilde{g}(x)d\delta_t(x) \label{eq:pote} \end{equation}
 for any \(g\in \cF_e\) and its quasi-continuous version \(\tilde{g}\). We have \(s(U\delta_t)=\{t\}\) because \(\cE(U\delta_t,g)=g(t)\).

Noting that \((\cE,\cF)\) is local, we see that \(\sigma((0,t])\) is conditionally independent of \(\sigma([t,\infty ))\) given \(\sigma (\{t\})\) by Theorem \ref{mainmain}. Since \(s(U\delta_t)=\{t\}\), \(X_{U\delta_{s}}\) is \(\sigma((0,t])\)-measurable for \(0<s\leq t\) and \(X_{U\delta_{t+r}}\) is \(\sigma([t,\infty ))\)-measurable for \(r\geq 0.\) 
\begin{proposition}
It holds that \(\sigma(\{t\})=\sigma(X_{U\delta_{t}})\vee \sigma(\mathscr{N})\).
\end{proposition}
\begin{proof}
Recall that, by the definition, \(\sigma(\{t\})\) contains \(\mathscr{N}\). By \((\ref{eq:pote})\), we have \(\sigma(\{t\})\supset \sigma(X_{U\delta_{t}})\vee \sigma(\mathscr{N})\).

We next prove \(\sigma(\{t\})\subset \sigma(X_{U\delta_{t}})\vee \sigma(\mathscr{N})\) by using the method of one-point extensions in \cite[section 7]{CF}. Let \(E^0:=E\setminus \{t\}\), \(Z^0\) be the part process of \(Z\) on \(E^0\) and \((\cE^0,\cF^0)\) be a Dirichlet form on \(L^2(E^0;m)\) associated with \(Z^0\). Then \(Z^0\) admits no killing inside and \(\varphi(x):=\mathbb{P}^0_x(\zeta^0<\infty, Z^0_{\zeta^0-}=t)\) is positive for any \(x\in E^0\), where \(\mathbb{P}^0_x\) is a distribution of \(Z^0\) starting at \(x\) and \(\zeta^0\) is a lifetime of \(Z^0.\) By \cite[Theorem 7.5.4]{CF},  \(\cF_e\) is spanned by \(\cF^0_e\) and \(\varphi.\)

Take \(f\in \cF_e\) with \(s(f)=\{t\}\). For \(g\in \cF_e\), there exists \(g^0\in \cF_e^0\) and a constant \(c^0\) such that \(g=g^0+c^0\varphi\), and we have 
\begin{eqnarray*}
\cE(f,g) &=& \cE(f,g^0)+\cE(f,c^0\varphi)\ =\ \cE(f,c^0\varphi)\\
&=&\frac{\cE(f,\varphi)}{\varphi(t)}c^0\varphi(t)\ =\ \frac{\cE(f,\varphi)}{\varphi(t)}g(t).
\end{eqnarray*}
Hence we have \(f=cU\delta_{t}\) q.e. for \(c:=\cE(f,\varphi)/\varphi(t)\), and \(\sigma({t})\subset \sigma(\{X_{cU\delta_{t}}\}_{c\in \mathbb{R}})\) holds. For \(c \in \mathbb{R}\) and \(u,v \in \cF_e\), we have \(\E((X_{cu}-cX_u)X_v))=0\) and so \(X_{cu}=cX_u\) a.s. Thus it holds that \(\sigma(\{X_{cU\delta_{t}}\}_{c\in \mathbb{R}})=\sigma(X_{U\delta_{t}})\) and \(\sigma(\{t\})\subset \sigma(X_{U\delta_{t}})\vee \sigma(\mathscr{N})\).
\end{proof}
Thus \(X_{U\delta_{s}}\) is conditionally independent of \(X_{U\delta_{t+r}}\) given \(\sigma(X_{U\delta_{t}})\) for \(0<s\leq t\) and \(r\geq 0\). This is the Markov property for \(\{X_{U\delta_t}\}_{t>0}\) as a stochastic process.
\end{example}

\begin{remark}
Consider the following special case in Example \ref{exampleBM};
\begin{eqnarray*}
\left\{ \begin{split} \cF:=H_0^1(E):= \left\{f\in L_{\rm loc}^2(\mathbb{R})\ ;\ \frac{d f}{d x}\in L^2(\mathbb{R}), f=0 \ {\rm q.e.\  on\ } \mathbb{R}\setminus E\right\},\\
\cE(f,g):=\frac{1}{2}\int_{\mathbb{R}} \frac{df}{dx}(x) \frac{dg}{dx}(x) dx \ \ {\rm for\ } f,g\in \cF. \hspace{35mm} \end{split}\right.
\end{eqnarray*}
By \cite[Proposition 2.13]{S}, \(\G(\cE)\) is the Gaussian free field on \(E=(0,\infty)\subset \mathbb{R}\).
By some standard computation, we have \(\cE(U\delta_{t/2},U\delta_{s/2})=t\wedge s\) and there exists a modification \(W\) of \(\{X_{U\delta_{t/2}}\}_{t>0}\) having continuous paths, which has the same law as that of Brownian motion.
\end{remark}

\begin{remark}
We can consider the multidimensional version of the state space of Example \ref{exampleBM}. Let \(E\subset \mathbb{R}^2\) be a domain, \(m\) be a positive Radon measure with \({\rm supp}(m)=E\), \((\cE,\cF)\) be an irreducible regular transient local Dirichlet form on \(L^2(E;m)\), \(Z\) be the Hunt process associated with \((\cE,\cF)\), and \(\G(\cE)\) be a Gaussian field associated with \((\cE,\cF)\). 
Fix \(z\in E\) and a continuous decreasing function \(r:(0,\infty)\to (0,\infty)\). \(\mu_{t}\) denotes a uniform measure on the circle of radius \(r(t)\) around \(z\), and assume that \(\mu_t\) is a measure of finite energy integral for all \(t>t_0:=r^{-1}({\rm dist}(z,\partial E))\). For fixed \(t,\) by Theorem \ref{mainmain}, \(\sigma(A)\) is conditionally independent of \(\sigma(\overline{A^c})\) given \(\sigma(\partial A)\) where \(A\) is the closed ball of radius \(r(t)\) around \(z\). It is easy to check that \(X_{U\mu_{s}}\) is \(\sigma(\overline{A^c})\)-measurable and \(X_{U\mu_{t+r}}\) is \(\sigma(A)\)-measurable for \(t_0 \leq s\leq t\) and \(r\geq 0.\) However, \(\sigma(\partial A)\) is not equal to \(\sigma(X_{U\mu_{t}})\) because it holds that \(\sigma(\partial A)\supsetneq \sigma(X_{U\mu})\) for \(\mu\), a non-uniform measure on the circle of radius \(r(t)\) around \(z\). Therefore we can not obtain the Markov property for \(\{X_{U\mu_t}\}_{t>t_0}\) directly from that of \(\G(\cE)\) in the multidimensional case by this way. However there is a case that \(\{X_{U\mu_t}\}_{t>t_0}\) has the Markov property. For example, for GFF on a domain of \(\mathbb{R}^2\) and \(r(t)=\exp{(-t)},\) the continuous modification of \(\{X_{U\mu_t}\}_{t>t_0}\) has the same law as Brownian motion (\cite[Theorem 1.35]{BP}).
\end{remark}

\section*{Acknowledgement}
I would like to thank Professor Takashi Kumagai for helpful discussions and comments. I appreciate Professor David A. Croydon and Professor Naotaka Kajino for checking the introduction of this paper. I also thank to Professor Masatoshi Fukushima for useful comments on how the Markov property is used (Remark \ref{Fukurem}). This work was supported by JSPS KAKENHI Grant Number JP21J20251.

\end{document}